\documentclass[12pt]{amsart}
\pdfoutput=1

\usepackage{amsmath,amssymb,amsthm}  
\usepackage{mathtools}
\usepackage{tikz}
\usepackage{comment}

\usepackage{xcolor} 	
\usepackage[unicode]{hyperref}
\hypersetup{
	colorlinks,
    linkcolor={red!60!black},
    citecolor={green!60!black},
    urlcolor={blue!60!black},
}
\usepackage[abbrev, msc-links]{amsrefs} 

\usepackage[utf8]{inputenc}
\usepackage[T1]{fontenc}
\usepackage{lmodern}
\usepackage[babel]{microtype}
\usepackage[english]{babel}

\usepackage{geometry}
\usepackage{enumitem}

\let\polishlcross=\l
\def\l{\ifmmode\ell\else\polishlcross\fi}

\let\emptyset=\varnothing

\let\theta=\vartheta
\let\rho=\varrho
\let\phi=\varphi

\def\NN{\mathbb N}

\newcommand{\script}{\mathcal}
\newcommand{\parentheses}[1]{{\left( {#1} \right)}}

\newcommand{\p}{\parentheses}

\newcommand{\Set}[1]{{\left\lbrace {#1} \right\rbrace}}

\def\set#1:#2{\Set{{#1} \colon {#2}}}
\newcommand{\col}[1]{\operatorname{col}\p{ {#1}}}

\DeclareFontFamily{U}  {MnSymbolC}{}
\DeclareSymbolFont{MnSyC}         {U}  {MnSymbolC}{m}{n}
\DeclareFontShape{U}{MnSymbolC}{m}{n}{
    <-6>  MnSymbolC5
   <6-7>  MnSymbolC6
   <7-8>  MnSymbolC7
   <8-9>  MnSymbolC8
   <9-10> MnSymbolC9
  <10-12> MnSymbolC10
  <12->   MnSymbolC12}{}
\DeclareMathSymbol{\powerset}{\mathord}{MnSyC}{180}

\theoremstyle{plain}
\newtheorem{thm}{Theorem}[section]

\newtheorem{prop}[thm]{Proposition}

\newtheorem{cor}[thm]{Corollary}

\theoremstyle{definition}

\title{A Cantor-Bernstein-type theorem for spanning trees in infinite graphs}

\author{Joshua Erde}
\author{Pascal Gollin}
\author{Atilla Jo\'{o}}
\thanks{The third author acknowledges support by the Alexander von Humboldt Foundation and partially by OTKA 129211.}
\author{Paul Knappe}
\author{Max Pitz}
\address{University of Hamburg, Department of Mathematics, Bundesstra{\ss}e 55 (Geomatikum), 20146 Hamburg, Germany}

\email{joshua.erde@uni-hamburg.de}
\email{pascal.gollin@uni-hamburg.de}
\email{attila.joo@uni-hamburg.de}
\email{paul.knappe@studium.uni-hamburg.de}
\email{max.pitz@uni-hamburg.de}


\begin{document}

\begin{abstract}
We show that if a graph admits a packing and a covering both consisting of $\lambda$ many spanning trees, where $\lambda$ is some infinite cardinal, then the graph also admits a decomposition into $\lambda$ many spanning trees. For finite $\lambda$ the analogous question remains open, however, a slightly weaker statement is proved.
\end{abstract}

\maketitle


\section{Introduction}\label{sec:introduction}
The graphs in this paper may have parallel edges but not loops. 
A \emph{spanning tree} of a graph~$G$ is a connected, acyclic subgraph $T \subseteq G$ containing all vertices of~$G$. 
Given a cardinal~$\lambda$, a \emph{$\lambda$-packing (of~$G$)} is a collection of $\lambda$ many edge-disjoint spanning trees in~$G$, a \emph{$\lambda$-covering (of~$G$)} is a collection of~$\lambda$ many spanning trees whose union covers the edge set of~$G$, and a \emph{$\lambda$-decomposition (of~$G$)} is a collection of~$\lambda$ many spanning trees whose edge sets partition the edge set of~$G$.

The purpose of this note is to establish the following Cantor-Bernstein-type theorem for decomposing infinite graphs into spanning trees:

\begin{thm}\label{thm:characterisation}
Let $\lambda$ be an infinite cardinal. Then a graph admits a $\lambda$-decomposition if and only if it admits both a $\lambda$-packing and a $\lambda$-covering.
\end{thm}

Perhaps interestingly, the $\lambda$ in Theorem \ref{thm:characterisation} does not need to be unique: For example, it is not hard to show directly that $K_{\aleph_1}$, the complete graph on $\aleph_1$ vertices, admits decompositions both into $\aleph_0$ or $\aleph_1$ many spanning trees. This effect can get arbitrarily pronounced, see Proposition~\ref{prop_example} below.

Our proof of Theorem~\ref{thm:characterisation} relies on two  well-known characterisations of when $G$ admits a $\lambda$-packing or $\lambda$-covering for an infinite cardinal $\lambda$. Firstly, for $\lambda$-packings, we have the following characterisation in terms of the edge-connectivity of $G$.

  \begin{thm}[Laviolette, {\cite{laviolette2005decompositions}*{Corollary 14}}]
 \label{thm:laviolette}
 Let $\lambda$ be an infinite cardinal. Then a graph admits a $\lambda$-packing if and only if it has edge-connectivity at least $\lambda$.
\end{thm}

The analogous statement for finite $\lambda$ fails dramatically: there are infinite graphs of arbitrarily large finite edge-connectivity which do not even contain two edge-disjoint spanning trees, see~\cite{aharoni1989infinite}.

Theorem~\ref{thm:laviolette} was originally obtained by Laviolette as corollary to his theory on ``bond-faithful decompositions'' which required the generalised continuum hypothesis (GCH). The use of GCH to obtain these bond-faithful decompositions was subsequently removed by Soukup \cite{soukup2011elementary}*{Theorem~6.3} using the technique of elementary submodels. In Section \ref{s:short}, we will give a short direct proof of Theorem \ref{thm:laviolette}, not relying on the  ``bond-faithful decomposition'' result.
 
The characterisation of the existence of $\lambda$-coverings relies on the following notion introduced by Erd\H{o}s and Hajnal \cite{erdHos1967decomposition}, which we adapt here slightly to take parallel edges into account: The \emph{colouring number} $\col{G}$ of a graph $G=(V,E)$ is the smallest cardinal $\mu$ such that there exists a well-ordering $<^*$ of $V$ such that for every $v\in V $ the cardinality of the set of edges between $v$ and $\set{w \in V}:{ w <^* v} $ is strictly less than $\mu$.  We call any well-ordering $<^*$ that witnesses the colouring number of a graph \emph{good}. The relation of the colouring number to $\lambda$-coverings is the following:

 \begin{thm}[Erd\H{o}s and Hajnal, {\cite{erdHos1967decomposition}*{Theorem 9}}]
 \label{thm:EH}
 \label{t:EHcover}
 Let $\lambda$ be an infinite cardinal. Then a graph admits a $\lambda$-covering if and only if it is connected and  has colouring number at most $\lambda^+$.
\end{thm}

The original proof of Theorem \ref{t:EHcover}, stated only for simple graphs, is quite oblique; it is reduced to a claim in an earlier paper by the same authors \cite{EH66}, the proof of which in turn is omitted, stating only that it follows from similar methods as a proof of Fodor, which itself is not entirely elementary. 

For this reason, we will also provide a short  proof of Theorem \ref{t:EHcover} in Section \ref{s:short}. Our proof has the additional feature that as a byproduct it yields that every graph has a good well-order of the shortest possible order type, $|V(G)|$. Previously this had to be deduced from Theorem~\ref{t:EHcover} together with a result of Erd\H os and Hajnal in \cite{EH66}*{Theorem 8.6}, or by employing the main theorem in \cite{bowler2015colouring} which characterises the colouring number of a simple graph in terms of forbidden subgraphs.

The structure of the paper is then as follows. In Section \ref{s:short} we provide short proofs of Theorems \ref{thm:laviolette} and \ref{t:EHcover}. In Section \ref{s:graphs} we prove Theorem \ref{thm:characterisation}, and finally in Section \ref{s:open} we discuss an open problem, namely whether Theorem \ref{thm:characterisation} also holds for finite $\lambda$.

\section{Elementary proofs of Laviolette and Erd{\H{o}}s-Hajnal}\label{s:short}
In this section, we provide elementary proofs of Theorems~\ref{thm:laviolette} and \ref{t:EHcover}. 

\begin{proof}[Proof of Theorem \ref{thm:laviolette}]
The forward implication is trivial. For the converse, consider a graph~$G$ of infinite edge-connectivity~$\lambda$. Let $V(G) = \set{v_j}:{j < \kappa}$. We will construct a family ${\script{T}=(T_i \colon i < \lambda)}$ of edge-disjoint spanning subgraphs (which will then  contain the desired trees) in $\kappa$ many steps as follows: for each $t < \kappa$, we find families $\script{T}_t=(T_i(t) \colon i < \lambda)$ of edge-disjoint connected subgraphs of $G$, all on the same vertex set $V_t \subset V$ which satisfies $\set{v_j}:{j < t} \subseteq V_t$. Moreover, we make sure that for every $i< \lambda$ we have $T_i(t) \subseteq T_i(t')$  whenever~${t < t'}$. Taking $T_i = \bigcup_{t < \kappa} T_i(t)$ yields the desired family $\script{T}$.

It remains to describe the construction. Initially we let $V_0 = \emptyset$. In a limit step we may simply take unions. At a successor step, suppose that in some step $t < \kappa$ the family $\script{T}_t$ is already defined. If $v_t \in V_t$, let $\script{T}_{t+1} = \script{T}_t$.
Otherwise, consider the graph~$G_t$ where we contract~$V_t$ to a single vertex~$x_t$ and delete all resulting loops. Since~$G$ has edge-connectivity~$\lambda$, so does~$G_t$. Hence, by greedily adding new paths, we can find a sequence $(S_k \colon k < \lambda)$ of edge-disjoint, connected subgraphs of $G_t$, all of size strictly less than $\lambda$, such that $x_t,v_t \in S_0$ and $V(S_k) \subseteq V(S_{k'})$ whenever $k < k'$. Let $V'_t := \bigcup_{k < \lambda} V(S_k)$. Next, partition $\lambda$ into $\lambda$ many subsets $(O_i \colon i < \lambda)$ each of cardinality $\lambda$, and define $H_i = \bigcup_{k \in O_i} S_k$, a
connected subgraph of~$G_t$ with vertex set~$V'_t$. If for each $i < \lambda$ we let $T_i(t+1)$ be the subgraph of $G$ with vertex set $V_{t+1} := V_t \cup (V'_t \setminus \Set{x_t})$ and edge set $E(T_i(t)) \cup E(H_i)$, then $\script{T}_{t+1}$ is as desired.
\end{proof}


\begin{proof}[Proof of Theorem \ref{t:EHcover}]
If the colouring number of $G$ is less than $\lambda^+$, then, following Erd\H{os}s and Hajnal, we can decompose $G$ into forests in the following manner: Let $(v_i \colon i < \kappa)$ be a good well-order of $V(G)$, i.e.\ one where for each $i$ the set $E_i$ of `backwards edges' from $v_i$ (edges between $v_i$ and some $v_j$ where $j<i$) has cardinality at most $\lambda$. For each $i < \kappa$ let us pick an arbitrary injection $f_i : E_i \rightarrow \lambda$ and for each $k < \lambda$ let $T_k = \bigcup_{i < \lambda} f_i^{-1}(k)$. In words, for each $i$ we pick an arbitrary rainbow colouring of $E_i$ with (at most)  $\lambda$ many colours, and then consider the monochromatic edge sets. Since $\bigcup_{i < \kappa} E_i = E(G)$, the family $(T_k \colon k < \lambda)$ covers all edges of~$G$. To see that each $T_k$ is a forest, note that every cycle $C$ in $G$ has a vertex $v_i \in V(C)$ of maximal index $i$. This, however, implies $|C \cap E_i| = 2$, and so $C \not\subseteq T_k$ for any $k$. Finally, since $G$ is connected, each forest can be extended to a spanning tree, and hence $G$ admits a $\lambda$-covering.

For the converse implication, suppose there exists a family of $\lambda$ many spanning trees $(T_i \colon i < \lambda)$ which covers $E(G)$. First we note that there are at most $\lambda$ many parallel edges between any two vertices of $G$, since at most one such edge is in each $T_i$. If $|E(G)|\leq \lambda$ then any well-ordering of $V(G)$ witnesses that $\col{G} \leq \lambda^+$. Hence we may assume that $|E(G)| > \lambda$ which, by the previous comment, implies $|V(G)| > \lambda$. Let us root each $T_i$ arbitrarily and let $\leq_i$ be the corresponding tree order on $V(G)$, cf.~\cite{D16}*{\S1.5}. For a vertex $x$, recall that $\lceil x \rceil_i = \set{v}:{v \leq_i x}$ denotes the vertex set of the path from the root to $x$ in $T_i$. Consider the following closure operation of a given vertex set $X\subseteq V(G)$: 
Let $X_0 = X$ and for each $n \in \NN$ put $X_{n+1} :=
\bigcup \set{\lceil x \rceil_i}:{x \in X_n, \; i < \lambda}$.

Let cl$(X) = \bigcup_{n \in \NN} X_n$ be the \emph{closure} of $X$. We say a set $Y\subseteq V(G)$ is \emph{closed} if cl$(Y)=Y$, and it is clear that cl$(X)$ is closed for every $X \subseteq V$. Since there are only~$\lambda$ many trees~$T_i$, and~${\lceil x \rceil_i}$ is finite for each~$i$, it follows that whenever $X$ is closed and $Y \supseteq X$ is such that $|Y \setminus X| \leq \lambda$ then there is a closed set $Z \supseteq Y$ with $|Z \setminus X| \leq \lambda$.

Now let $(v_i \colon i < \kappa)$ be a well-ordering of $V(G)$ of length $\kappa=|V(G)|$ and define an increasing sequence of closed sets $(V_i \colon i < \kappa)$ by $V_0 =\emptyset$, $V_{i+1} =$ cl$\{V_i \cup \{\min_j \{v_j \colon v_j \not\in V_i \} \}\}$ for each~${i < \kappa}$, and $V_i = \bigcup_{j<i} V_j$ for $i < \kappa$ a limit. In particular, we have $\bigcup_{i<\kappa} V_i = V(G)$ and $|V_{i+1} \setminus V_i| \leq \lambda$ for each $i < \kappa$. Let us well-order each set $V_{i+1} \setminus V_i$ arbitrarily, and concatenate these orderings to form a well-order~$<^*$ of~$V$. We claim that this well-ordering of order type~$|V(G)|$ witnesses col$(G) \leq \lambda^+$. Indeed, let $v\in V$ be arbitrary. There is a unique $i$ such that~${v \in V_{i+1} \setminus V_i}$, and hence every `backwards edge' (with respect to~$<^*$) from $v$ has both endpoints in $V_{i+1}$. We will show that there at at most $\lambda$ many such edges.

Firstly, since $|V_{i+1} \setminus V_i| \leq \lambda$, there are at most $\lambda\cdot \lambda=\lambda$ many edges between $V_{i+1} \setminus V_i$ and $v$. Furthermore, suppose $e = (x,v)$ is an edge between $V_i$ and $v$. There is some $j$ such that $e \in E(T_j)$ and, since $V_i$ is closed under the tree-order generated by any~$T_{j}$ and $v \not\in V_i$, it follows that $x \leq_j v$. However, there is a unique edge $(x,v) \in E(T_j)$ such that $x \leq_j v$. It follows that there are at most $\lambda$ many edges between $V_i$ and $v$
\end{proof}

We remark that only the backwards implication used that $\lambda$ is infinite.

\begin{cor}
Every graph has a good well-ordering of order-type~$|V(G)|$.
\qed
\end{cor}

\section{A Cantor-Bernstein theorem for spanning trees in infinite graphs}\label{s:graphs}

\begin{thm}\label{thm:Gdecomposition}
Let $\lambda$ be a cardinal (finite or infinite) and let $G$ be a graph with col$(G) \leq \lambda^+$ which admits $\lambda$-packing. Then $G$ admits a $\lambda$-decomposition.
\end{thm}
\begin{proof}
Let $(v_i \colon i < \kappa)$ be a good well-ordering of~$V(G)$. For each $i < \kappa$ let $E_i$ be the set $\{ (v_j,v_i) \in E(G) \colon j < i\}$ of `backwards edges' in this ordering at $v_i$. Then $(E_i \colon i < \kappa)$ is a partition of $E(G)$ and $|E_i| \leq \lambda$ for each $i < \kappa$. Let us well-order each of the sets $E_i$ arbitrarily in order type $|E_i|$ and concatenate these orderings to form a well-order $\prec$ of $E$. 

By assumption, there exists a family $(T_i \colon i <\lambda)$ of $\lambda$ many edge-disjoint spanning trees of~$G$. If $\bigcup_{i < \lambda} T_i = E(G)$, then $(T_i \colon i <\lambda)$ is a $\lambda$-decomposition. Our aim will be to exchange a yet uncovered edge $f \in E(G) \setminus \bigcup_{i < \lambda} T_i$ with some later edge $e \succ f$ from some $T_i$ such that at each stage in our process we maintain the property that $(T_i \colon i <\lambda)$ is a $\lambda$-packing. By an appropriate book-keeping procedure, we guarantee that each edge is eventually covered.

Let us initialise by setting $T_i(0) = T_i$ for each $i < \lambda$. Suppose that we have already constructed a $\lambda$-packing $\mathcal{T}_t = (T_i(t) \colon i < \lambda)$ where $t< \kappa$. In step $t$ we consider $e_t$. If $e_t \in \bigcup_{i<\lambda} E(T_i(t))$, then we set $T_i(t+1) = T_i(t)$ for each $i < \lambda$. Otherwise, $e_t \not\in \bigcup_{i<\lambda} T_i(t)$. Then $e_t \in E_i$ for some $i$ and by construction there are fewer than $\lambda$ many edges $e \in E_i$ such that $e \prec e_t$, and hence there is some $k < \lambda$ such that $T_k(t)$ contains no edges $e \in E_i$ with~${e \prec e_t}$. Since $T_k(t)$ is a spanning tree, there is a unique cycle $C \subseteq T_k(t) + e_t$. Since $C$ is finite, it contains a $\prec$-maximal edge $f$. Moreover, since $T_k(t)$ contains no edges $e \in E_i$ with $e \prec e_t$ it follows that $f \neq e_t$: if $j$ is maximal such that $C \cap E_{j} \neq \emptyset$ then $|C \cap E_{j}| = 2$, since~$C$ is a cycle. Then, if $j = i$ it follows that $e_t \prec f$ by our choice of $T_k(t)$ and if $j > i$ then clearly $e_t \prec f$ since all of $E_i$ precedes $E_j$.

Now let $T_k(t+1) = T_k(t) - f + e_t$, which is again a spanning tree, and $T_i(t+1) := T_i(t)$ for all $k \neq i < \lambda$. Finally for each limit ordinal $\tau < \kappa$ we let
\[
T_i(\tau) = \{ e \colon \text{ there exists } t_0 < \tau \text{ such that } e \in T_i(t) \text{ for all } t_0 < t < \tau \}
\]
We claim that for every $t \leq \kappa$ the family $\mathcal{T}_t$ is indeed a $\lambda$-packing. Since this property is clearly preserved at successor steps, it remains to check that it holds at limit steps. 

As it is clear that if each  $\mathcal{T}_t$ is a family of edge-disjoint subgraphs for $t < \tau$, then $\mathcal{T}_\tau$ is a family of edge-disjoint subgraphs, it is sufficient to show that each $T_i(\tau)$ is in fact a spanning tree. That each $T_i(\tau)$ is acyclic is clear, as any finite cycle would have to appear at some successor step. To see that $T_i(\tau)$ is connected and spanning, it suffices to show that it contains an edge from each bond of $G$. 

Given a bond $F \subset E(G)$  let us consider the set of edges $F_i(t) := E(T_i(t)) \cap F$. We claim that the sequence $f_i(t) := \min_\prec F_i(t)$ is $\prec$-non-increasing in $t$. Indeed, suppose we delete the $\prec$-minimal edge $f$ of $F_i(t)$ from $T_i(t) $ at step $t$. Note that by the construction there is a cycle $C$ with $\prec$-maximal edge $f$ such that $C-f \subset T_i(t+1)$. Then $C \cap F$ is non-empty because it contains $f$ and therefore, since $\left|C\cap F\right|$ must be even, there is some $e\neq f$ in~${C\cap F}$. It follows from the $\prec$-maximality of $f$ in $C$ that $e \prec f$. Furthermore, $e\in F_i(t+1)$ since $C-f \subset T_i(t+1)$, from which $f_i(t+1)\prec f_i(t)$ follows. Hence for each bond $F$ and each limit ordinal $\tau$, the sequence $(f_i(t) \colon t < \tau)$ is constant after some $t_0<\tau$, and therefore $f_i(t_0) \in F \cap T_i(\tau)$. 

It remains to verify that $\mathcal{T}_\kappa$ is a $\lambda$-decomposition. Since it is a $\lambda$-packing by the above, it suffices to show that $\bigcup_{i < \lambda} E(T_i(\kappa)) = E(G)$. However for each $t < \lambda$ we have $e_t \in E(T_k(t+1))$ for some $k$ by construction. Furthermore, at any later stage $s$ we only ever remove an edge $f$ with $e_t \prec e_s \prec f$. It follows that $e_t \in E(T_k(s))$ for all $s > t$ and hence $e_t \in E(T_k(\kappa))$. 
\end{proof}

Theorem \ref{thm:characterisation} then follows from Theorems \ref{thm:Gdecomposition} and \ref{t:EHcover}. We conclude this section by observing that the effect of a graph having $\lambda$-decompositions for different $\lambda$'s can get arbitrarily pronounced:

\begin{prop}
\label{prop_example}
For every infinite cardinal $\kappa$ there is a graph that admits a $\lambda$-decomposition for any choice of $\lambda$ with $2 \leq \lambda \leq \kappa$. 
\end{prop}

\begin{proof}[Construction]
We construct the desired graph $G$ as an increasing union of graphs $G_n=(V_n,E_n)$ by recursion on $n \in \NN$ as follows.

Let $G_0=K_2$ be the complete graph on two vertices. We form $G_{n+1}$ by adding $\kappa$ many new $u-v$ paths of length two to $G_n$ for every $u\neq v \in V_n$, internally disjoint from each other and from $V_n$. Finally, we set $G:= \bigcup_{n\in \NN} G_n$ which by construction has (edge-)connectivity~$\kappa$. If we well-order each $V_{n+1}\setminus V_n$ arbitrarily and concatenate these orders, we obtain a well-ordering witnessing $\col{G}=3$, as by construction, every newly added vertex in step $n$ has degree two. Since $ \kappa$ was infinite, it follows from Theorem~\ref{thm:laviolette} that $G$ has a $\kappa$-packing, and hence a $\lambda$-packing for all $\lambda \leq \kappa$. Therefore, the assertion of the proposition follows from Theorem~\ref{thm:Gdecomposition}.
%
\end{proof}

\section{An open problem}\label{s:open}

It remains an interesting question whether the assertion of our main theorem also holds for finite $\lambda$. For finite graphs, a simple counting argument (every spanning tree has precisely $|G|-1$ edges) shows that Theorem~\ref{thm:characterisation} holds when both the graph and $\lambda$ are finite. Hence, the question remains what happens for infinite graphs and finite~$\lambda$. We note that our main technical result, Theorem~\ref{thm:Gdecomposition}, did not require that~$\lambda$ is infinite. However, in order to deduce Theorem \ref{thm:characterisation} from it we needed to apply Theorem~\ref{t:EHcover}, which only holds for infinite~$\lambda$. When~$\lambda$ is finite, only the following, slightly weaker version of Theorem \ref{t:EHcover} holds, which is best possible as can be seen in the case of complete graphs.

\begin{thm}[{Erd\H{o}s and Hajnal, \cite{erdHos1967decomposition}*{Theorem~11}}]
\label{thm:covering implies colering}
If $G$ is a graph (finite or infinite) with a $k$-covering for some $k \in \NN$, then col$(G) \leq 2k$.
\end{thm}

The following is then a consequence of Theorems \ref{thm:Gdecomposition} and \ref{thm:covering implies colering}.

\begin{cor}
For every $k\geq 1$, every graph with a $k$-covering and a $(2k-1)$-packing has a $(2k-1)$-decomposition.
\end{cor}



Hence, if one were to seek a proof for Theorem~\ref{thm:characterisation} for finite $k=\lambda$, one would need to use the assumption of the existence of a $k$-covering more efficiently than simply relying on the rather weak consequence that col$(G) \leq 2k$. One such possibility might be offered by the following characterisation due to Nash-Williams (where the assertion for infinite graphs follows from the finite version by a straightforward compactness argument):

\begin{thm}[Nash-Williams \cite{NW64}]
\label{t:NWcovering}
For every $k\in \mathbb{N}$, a graph $G$ admits a $k$-covering if and only if for every non-empty finite $U \subseteq V(G)$ the number of edges in $G[U]$ is at most $k(|U|-1)$.
\end{thm}

However, we did not succeed in proving a theorem in the vein of Theorem~\ref{thm:Gdecomposition} using Nash-Williams's condition.

Finally, we remark that in order to prove the assertion of Theorem~\ref{thm:characterisation} for finite $\lambda=k$, it suffices to consider countable graphs: Indeed, to see that the general case follows from the countable case, consider some uncountable graph $G$ with a $k$-packing  $\{T_1,\ldots,T_k\}$ and a $k$-covering $\{T_{k+1},\ldots,T_{2k}\}$. Starting with $W_0 = \emptyset$, by greedily adding finite paths from the different trees in turn for $\omega$ many substeps, we find an increasing, continuous collection $(W_i \colon i < |G|)$ of subsets of $V$ with $\bigcup_i W_i = V$ such that $W_{i+1} \setminus W_i$ is countable, and each $T_j[W_i]$ is an induced subtree of $T_j$ for all $j \in [2k]$ and $i < |G|$. Then each minor $G_i=G[W_{i+1}]/G[W_i]$   has a $k$-packing and $k$-covering given by the trees $T_j[W_{i+1}]/T_j[W_i]$. Applying the countable assertion to each $G_i$ yields a $k$-decomposition $\{S_1(i),\ldots,S_k(i)\}$ of~$G_i$. Clearly, the subtrees~$S_j$ of~$G$ for $j \in [k]$ given by $E(S_j)  = \bigcup_{i} E(S_j(i))$ are as desired.

\bibliographystyle{unsrtnat}
\bibliography{references}

\end{document}